\documentclass{amsart}

\usepackage{amssymb, amsmath, amsthm, units}
\usepackage{verbatim}
\usepackage{comment}
\usepackage{graphicx}
\usepackage{enumerate}
\usepackage{color}
\usepackage{hyperref}
\usepackage{cleveref}

\DeclareGraphicsExtensions{.pdf,.jpg,.png}

\theoremstyle{plain}
\newtheorem{theorem}{Theorem}
\newtheorem{lemma}[theorem]{Lemma}

\newtheorem{corollary}[theorem]{Corollary}
\newtheorem{fact}[theorem]{Fact}

\newtheorem{conjecture}[theorem]{Conjecture}

\theoremstyle{definition}

\newcommand{\cE}{\mathcal{E}}
\newcommand{\cL}{\mathcal{L}}

\newcommand{\cV}{\mathcal{V}}

\newcommand{\eps}{\varepsilon}

\makeatletter

\begin{document}

\title{The Brown-Erd\H os-S\'os Conjecture in finite abelian groups}

\author{J\'{o}zsef Solymosi}
\address{\noindent Department of Mathematics, University of British Columbia, Vancouver, BC, Canada V6T 1Z2}
\email{solymosi@math.ubc.ca}

\author{Ching Wong}
\address{\noindent Department of Mathematics, University of British Columbia, Vancouver, BC, Canada V6T 1Z2}
\email{ching@math.ubc.ca}

\date{}

\begin{abstract}
	The Brown-Erd\H{o}s-S\'{o}s conjecture, one of the central conjectures in extremal combinatorics, states that for any integer $m\geq 6,$ if a 3-uniform hypergraph on $n$ vertices contains no $m$ vertices spanning at least $m-3$ edges, then the number of edges is $o(n^2).$ We prove the conjecture for triple systems coming from finite abelian groups.
	\end{abstract}

\maketitle

\section{Introduction}

In extremal graph theory, a fundamental problem is to determine the maximum number of edges in a graph on $n$ vertices such that the graph does not contain certain subgraphs. In 1973, Brown, Erd\H{o}s and S\'{o}s \cite{BrownErdosSos73}  studied the problem in $3$-uniform hypergraphs on $n$ vertices, forbidding all sub-hypergraphs with $m$ vertices and $k$ hyperedges. The maximum number of hyperedges in such hypergraphs is denoted by $f(n,m,k)$. They determined the asymptotic behaviour of $f(n,m,k)$ for most $k$ when $m\leq6$. Moreover, the bound $f(n,m,m-2)=\Theta(n^2)$ is established for any fixed $m \geq 4$ using randomness. They conjectured that $f(n,m,m-3)=o(n^2)$ for any fixed $m\geq6$, which is now known as the Brown-Erd\H{o}s-S\'{o}s conjecture.

Ruzsa and Szemer\'{e}di \cite{RuzsaSzemeredi1978} later resolved the $(6,3)$-problem by proving that $f(n,6,3) = o(n^2)$. Whether it is true that $f(n,7,4)=o(n^2)$ is still an open problem. The first author provided a partial answer to the $(7,4)$-problem in \cite{Solymosi15}. It was observed that the Brown-Erd\H{o}s-S\'{o}s conjecture is equivalent to the following:

\begin{conjecture}[Brown-Erd\H{o}s-S\'{o}s]\label{conj:BES}
	Fix $m \geq 6$. For every $c>0$, there exists a threshold $N = N(c)$ such that if $A$ is a quasigroup with $|A| = n \geq N$, then for every set $S$ of triples of the form $(a,b,ab) \in A^3$ with $|S| \geq cn^2$, there exists a subset of $m$ elements of $A$ which spans at least $m-3$ triples of $S$.
\end{conjecture}

When $m=7$, the first open case, Solymosi proved the validity of the above conjecture for finite groups $A$. In fact, the following quantitative version is proved, for finite abelian groups $A$. An \emph{$(m,k)$-configuration} is a set of $k$ triples $(a,b,ab)\in A^3$ on $m$ elements of $A$.

\begin{theorem}[Solymosi \cite{Solymosi15}]
	\label{thm:Solymosi15b}
	For every $c>0$, there exists a constant $\mu_1 = \mu_1(c) > 0$ such that if $A$ is a finite abelian group with $|A| = n$, then for every set $S$ of triples of the form $(a,b,ab) \in A^3$ with $|S| \geq cn^2$, 
	$S$ contains at least $\lfloor \mu_1 n^3 \rfloor$ $(7,4)$-configurations.
\end{theorem}

In this article we extend the above theorem in a much stronger form. We prove that every dense subset of triples $(a,b,ab)$ there are vertices spanning many edges. Asymptotically, there are $m$ vertices spanning at least $4/3m$ edges, which is much stronger than the conjectured amount. In order to prove this one should strongly use the group structure, since similar result is not true for general triple systems. In a recent breakthrough result Glock, K\"uhn, Lo, and Osthus \cite{GKLO}, and independently Bohman and Warnke \cite{BW}, proved a related conjecture of Erd\H os asymptotically. They proved that there are almost complete Steiner triple systems without containing $m$ points with $m-2$ triples. (Triple systems with this sparseness property are also referred to as having high girth.) Both works use randomness, an advanced extension of the triangle removal process. In the next theorem we show that structure enforces high local density.

\begin{theorem}
	\label{thm:mu}
	For every $c>0$ and integer $t \geq 2$, there exists a constant $\mu_t = \mu_t(c) > 0$ such that if $A$ is a finite abelian group with $|A| = n$, then for every subset $S_0$ of triples of the form $(a,b,ab) \in A^3$ with $|S_0| \geq cn^2$, 
	$S_0$ contains $\lfloor \mu_t n^3 \rfloor$ $\left(\nu,\frac{4(\nu-3t)}{3}(1-\frac{1}{4^t})\right)$-configurations, where, for each of the configurations, $2^{t+1} \leq \nu \leq 4^t + 3t$.
\end{theorem}

Observe that Theorem \ref{thm:mu} reduces to Theorem \ref{thm:Solymosi15b} when $t=1$. On the other hand, for large $t$ we obtain asymptotically $\left(\nu,\frac{4\nu}{3}\right)$-configurations. 
In particular, by the monotonicity $f(n,m,k)\leq f(n,m,k+1)$, Conjecture \ref{conj:BES} holds true for infinitely many values of $m$ for finite  abelian groups $A$.

One can also show that the statement above (without the quantitative part)  holds for non-abelian groups as well. 

\begin{corollary}\label{nonabel}
For every $c>0$ and integer $t \geq 2$, there exists a threshold $N_t = N_t(c)$ such that if $G$ is a finite group with $|G| = n \geq N_t$, then for every subset $S_0$ of triples of the form $(a,b,ab) \in G^3$ with $|S_0| \geq cn^2$, 
	$S_0$ contains a $\left(\nu,\frac{4(\nu-3t)}{3}(1-\frac{1}{4^t})\right)$-configuration, where $2^{t+1} \leq \nu \leq 4^t + 3t$.
	%For every $c>0$ and integer $t \geq 2$, there exists a constant $\mu_t = \mu_t(c) > 0$ such that if $G$ is a finite group with $|G| = n$, then for every subset $S_0$ of triples of the form $(a,b,ab) \in G^3$ with $|S_0| \geq cn^2$, $S_0$ contains $\left(\nu,\frac{4(\nu-3t)}{3}(1-\frac{1}{4^t})\right)$-configurations, where, for each of the configurations, $2^{t+1} \leq \nu \leq 4^t + 3t$.
\end{corollary}

The above corollary follows from Theorem \ref{thm:mu} and the following result of Pyber. 

\begin{theorem}[Pyber \cite{PY}]
There is a universal constant $\nu >0$ so that every group of
order $n$ contains an abelian subgroup of order at least $e^{\nu\sqrt{\log{n}}}.$
\end{theorem}

For the proof of Corollary \ref{nonabel} one should find a large abelian subgroup, $F\leq G,$ and follow the steps of the proof below using cosets $aF, Fb$ and $aFb,$ where the selected triples form a dense subset. The technique is similar to the one applied in \cite{Solymosi15}, but it would make the proof below much harder to follow. We decided to omit the proof.

One of our main tools is the following standard consequence of the regularity lemma of Szemer\'{e}di \cite{Szemeredi1975}, in which certain edges --- those incident to $V_0$, between irregular pairs, or between regular pairs having density at most $2\eps$ --- are removed. We are going to iterate the regularity lemma using it in a way similar to standard applications, like in \cite{KoSiSze}.

\begin{theorem}[Szemer\'{e}di \cite{Szemeredi1975}]
	\label{thm:regularitylemma}
	For each $\eps >0$, there exist integers $K = K(\eps)$ and $n_0 = n_0(\eps)$ such that for every bipartite graph $G$ with $n \geq n_0$ vertices in each bipartition ($A_1,A_2$) and at least $24 \eps n^2$ edges, there is a partition of the vertex set $V(G) = A_1 \cup A_2 = V_0 \cup V_1 \cup \cdots \cup V_{2k}$ and a subgraph $G'$ of $G$ satisfying the following conditions:
	\begin{enumerate}[(i)]
		\item $\frac 1 \eps \leq 2k \leq K$
		\item $V_i \subset A_1$ and $V_{k+i} \subset A_2$ for all $1 \leq i \leq k$
		\item $|V_0| \leq 2 \eps n$ and
		\[
		\dfrac n k (1 - \eps) \leq |V_1| = \cdots = |V_{2k}| \leq \dfrac n k
		\]
		\item $|E(G')| \geq |E(G)|/2$
		\item the graph $G'$ has no edges incident to $V_0$
		\item if there is an edge between $V_i$ and $V_{k+j}$ in $G'$, for some $1 \leq i, j \leq k$, then for subsets $U \subset V_i$ and $W \subset V_{k+j}$ with $|U| \geq \eps |V_i|$ and $|W| \geq \eps |V_{k+j}|$, the number of edges between $U$ and $W$ in $G'$ is at least $\eps |U||W|$.
	\end{enumerate}
\end{theorem}

\section{Some definitions}
In this section, $A$ is a finite abelian group and $S$ is a set of triples of the form $(a,b,ab) \in A^3$.

We define a bipartite graph $G_S$ with $2n$ vertices and $|S|$ edges that captures the triples of $S$: the vertex set of $G_S$ is $A_1 \cup A_2$, where $A_1$ and $A_2$ are copies of $A$, and $(a,b) \in A_1 \times A_2$ is an edge if and only if $(a,b,ab) \in S$.

We say that a quadruple $(a,b,c,d) \in A^4$ is \emph{$S$-good} if $a \neq c$, $ab = cd$, and $(a,b,ab),(a,d,ad),(c,b,cb),(c,d,cd) \in S$.

Note that every $S$-good quadruple in $A^4$ gives us a $(7,4)$-configuration in $S$. In \cite{Solymosi15}, a slightly stronger result than \cref{thm:Solymosi15b} was proved:

\begin{theorem}[Solymosi \cite{Solymosi15}]
	\label{thm:Solymosi15c}
	For every $c>0$, there exists a constant $\mu_1 = \mu_1(c) > 0$ such that if $A$ is a finite abelian group with $|A| = n$, then for every set $S$ of triples of the form $(a,b,ab) \in A^3$ with $|S| \geq cn^2$, there are at least $\lfloor \mu_1 n^3 \rfloor$ $S$-good quadruples in $A^4$.
\end{theorem}

For $\vec x = (x_1,x_2,x_3) \in A^3$, denote by $Q_S(\vec x)$ the set of $S$-good quadruples $(a,b,c,d)$ such that $x_1 = cb$, $x_2 = ab$ and $x_3 = ad$, and let $q_S(\vec x) = |Q_S(\vec x)|$.

Note that $q_S(\vec x) = 0$ if $x_1 = x_2$. We use group properties to obtain bounds for $q_S(\vec{x})$ and the number of triples with $q_S(\vec{x})\geq1$.

\begin{fact}
	\label{fact:a=a'}
	Let $(a,b,c,d)$ and $(a',b',c',d')$ be quadruples in $Q_S(\vec{x})$ for some $\vec{x} \in A^3$. If $a' = a$, then $(a',b',c',d') = (a,b,c,d)$.
	
	In particular, $q_S(\vec x) \leq n$ for all $\vec x \in A^3$.
\end{fact}

\begin{proof}
	Suppose that $a' = a$. By $ab = a'b'$ and $ad = a'd'$, we have $b = b'$ and $d = d'$. By $cb = c'b'$, we have $c = c'$.
\end{proof}

\begin{fact}
	\label{fact:q>=1}
	Let $(a,b,c,d)$ and $(a',b',c',d')$ be $S$-good quadruples with $ab = a'b'$. Then, $cb = c'b'$ if and only if $ad = a'd'$.
	
	In particular, the number of triples $\vec x$ of $A$ with $q_S(\vec x) \geq 1$ is at most $n(n-1) < n^2$.
\end{fact}

\begin{proof}
	Note that $cd = ab = a'b' = c'd'$.
	
	If $cb = c' b'$, then
	\[
	ad = a b (c b)^{-1} cd = a'b' (c'b')^{-1} c'd' = a'd'.
	\]
	Similarly, we have $ad = a'd'$ implies $cb = c'b'$.
\end{proof}

Let $(a,b,c,d)$ and $(a',b',c',d')$ be distinct quadruples of $A$. We say that they are \emph{disjoint} if $\{ a,c \} \cap \{ a',c' \} = \{ b,d \} \cap \{ b',d' \} = \emptyset$. In this case, the triples $(a,b,ab), (a,d,ad), (c,b,cb), (c,d,cd), (a',b',a'b'), (a',d',a'd'), (c',b',c'b'), (c',d',c'd')$ are all distinct. In our calculations, it will be much easier to estimate the number of triples in $S$ if the $S$-good quadruples considered are pairwise disjoint.

The following fact implies that at least a third of the $S$-good quadruples in $Q_S(\vec{x})$ are pairwise disjoint, given any $\vec{x} \in A^3$. Let $\tilde{Q}_S(\vec{x})$ be a subset of $Q_S(\vec{x})$ consisting of $\tilde{q}_S(\vec{x}) \geq q_S(\vec{x})/3$ pairwise disjoint $S$-good quadruples.

\begin{fact}
	\label{fact:share}
	Let $(a,b,c,d)$ and $(a',b',c',d')$ be distinct but not disjoint quadruples in $Q_S(\vec{x})$ for some $\vec{x} \in A^3$. Then, either one of the following hold:
	\begin{enumerate}
		\item $cb = ad$, $a' = c$, $c' = a$, $b' = d$ and $d' = b$, or
		\item $cb \neq ad$, $a' = c$, $c' \neq a$, $b' = d$ and $d' \neq b$, or
		\item $cb \neq ad$, $a' \neq c$, $c' = a$, $b' \neq d$ and $d' = b$.
	\end{enumerate}
	
	In particular, for each $(a,b,c,d) \in Q_S(\vec{x})$, all but almost 2 other quadruples in $Q_S(\vec{x})$ are disjoint from $(a,b,c,d)$.
\end{fact}

\begin{proof}
	By definition, the two quadruples are not disjoint means that $(\{ a,c \} \cap \{ a',c' \}) \cup( \{ b,d \} \cap \{ b',d' \}) \neq \emptyset$
	
	By \cref{fact:a=a'}, if either $a = a'$, $b = b'$, $c = c'$, or $d = d'$, then $(a',b',c',d') = (a,b,c,d)$. Hence, we may assume that none of these 4 equalities hold.
	
	Note that $a' = c$ if and only if $b' = d$, by $a'b' = ab = cd$. Similarly, $c' = a$ if and only if $d' = b$, by $c'd' = a'b' = ab$. 
	
	It is clear that $cb = ad$ in the first case, and that $cb \neq ad$ in the other 2 cases.
\end{proof}

\section{Proof of \cref{thm:mu} when $t = 2$} 
\label{section:t=2}

We fix a real number $c > 0$. Let $A$ be a finite abelian group with order $n$ and let $S_0$ be a set of triples of the form $(a,b,ab) \in A^3$ with $|S_0| \geq cn^2$.

In this section, we will find a subset $S_1$ of $S_0$, a triple $\vec{y}^{(1)} = (y_1^{(1)}, y_2^{(1)}, y_3^{(1)}) \in A^3$, a constant $c_1 > 0$ depending only on $c$, such that if $n$ is large enough, there are at least $\lfloor \mu_1(c_1) n^3 \rfloor$ $S_1$-good quadruples in $A^4$ and every $S_1$-good quadruple $(a,b,c,d)$ belongs to some $Q_{S_1}(\vec y^{(2)})$, where $\vec y^{(2)} = (y_1^{(2)}, y_2^{(2)}, y_3^{(2)}) \in A^3$ is element-disjoint from $\vec y^{(1)}$, and for some $a_i,b_i,c_i,d_i \in A$ (depending on $a,b,c,d$), the quadruples $R_1 = (a,b_1,c_1,d_1), R_2 = (a_2,b,c_2,d_2), R_3 = (c,b_3,c_3,d_3), R_4 = (a_4,d,c_4,d_4)$ are all in $\tilde Q_{S_0'}(\vec y^{(1)})$, for some $S_0' \subset S_0$. See Figure \ref{fig:(a,b,c,d)}.

\begin{figure}[ht!]
	\centering
	\includegraphics[width=120mm]{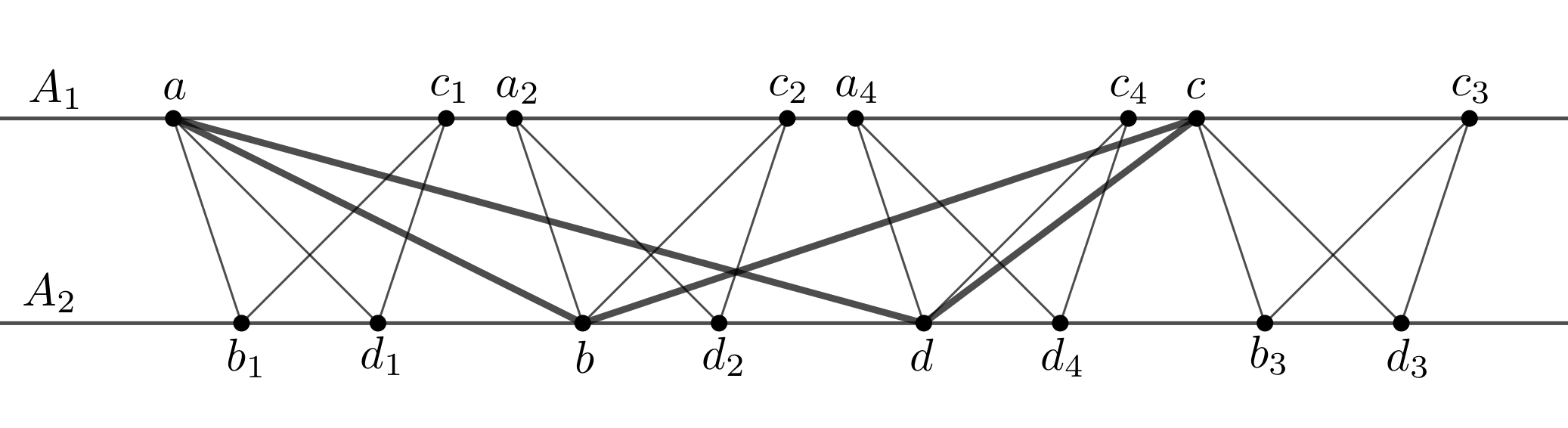}
	\caption{A subgraph of $G_{S_0'}$ corresponding to the 5 good quadruples $(a,b,c,d), R_1, R_2, R_3, R_4$.}
	\label{fig:(a,b,c,d)}
\end{figure}

We note that these 4 quadruples $R_i$ are distinct. Indeed, since $(a,b,c,d)$ is an $S_1$-good quadruple, by definition $a \neq c$, and therefore $b \neq d$, which implies $R_1 \neq R_3$ and $R_2 \neq R_4$. Moreover, if any two other quadruples are the same, say $R_1 = R_2$, then $y_2^{(1)} = ab_1 = ab = y_2^{(2)}$, contradicting the property that $\vec y^{(1)}$ and $\vec y^{(2)}$ are element-disjoint.

Consider the elements $\{a,b,c,d\} \cup \{a_i,b_i,c_i,d_i\} \cup \{y_j^{(1)},y_j^{(2)}\}$. There are $\nu \leq 4 + 12 + 6 = 22 = 4^2 + 3(2)$ elements. They span at least
\[
4+4^2 = \dfrac {4(16)} 3\dfrac{15}{16}\geq \dfrac {4(\nu - 6)}{3} \dfrac{15}{16}
\]
triples of $S_0$, as desired when $t = 2$.

By assumption, the graph $G_{S_0}$ contains at least $cn^2$ edges. We obtain a vertex partition $V_0 \cup V_1 \cup \cdots \cup V_{2k}$ and a subgraph $G_{S_0}'$ of $G_{S_0}$ using the regularity lemma (\cref{thm:regularitylemma}) with $\eps \leq \frac c {24}$ to be determined later (in \cref{lemma:eps}). Note that we can assume that $n \geq n_0(\eps)$ by choosing a smaller $\mu_2(c)$ if necessary at the end. By condition \emph{(iv)}, $G_{S_0}'$ contains at least $\frac c 2 n^2$ edges. One may identify the graph $G_{S_0}'$ with the graph $G_{S_0'}$, where $S_0'$ is a subset of $S_0$ with $|S_0'| \geq |S_0| / 2 \geq \frac{c}{2}n^2$.

\begin{lemma}
	\label{lemma:largestqvalue}
	There is a triple $\vec{y}^{(1)} \in A^3$ such that
	\[
	\tilde{q}_{S_0'}(\vec{y}^{(1)}) > \dfrac{\lfloor \mu_1(c/2) n \rfloor}{3}.
	\]
\end{lemma}

\begin{proof}
	We apply \cref{thm:Solymosi15c} with $S = S_0'$, we have $\lfloor \mu_1(c/2) n^3 \rfloor$ $S_0'$-good quadruples in $A^4$, i.e.
	\[
	\sum_{\substack{\vec x \in A^3 \\ q_{S_0'}(\vec x) \geq 1}} q_{S_0'}(\vec x) \geq \lfloor \mu_1(c/2) n^3 \rfloor.
	\]
	Let $\vec y^{(1)} \in A^3$ be a triple having the largest $q_{S_0'}$-value. By \cref{fact:q>=1}, we have
	\[
	q_{S_0'}(\vec y^{(1)}) > \dfrac{\lfloor \mu_1(c/2) n^3 \rfloor}{n^2} \geq \lfloor \mu_1(c/2) n \rfloor,
	\]
	and by \cref{fact:share},
	\[
	\tilde q_{S_0'}(\vec y^{(1)}) > \dfrac{\lfloor \mu_1(c/2) n \rfloor}{3}.
	\]
\end{proof}

We consider the vertices incident to the edges $(a,b) \in A_1 \times A_2$, where $(a,b,c,d) \in \tilde{Q}_{S_0'}(\vec{y}^{(1)})$, in the graph $G_{S_0'}$. We will show that there are at least $c_1n^2$ edges between these vertices in $G_{S_0'}$, for some $c_1 > 0$ depending on $c$. This allows us to apply \cref{thm:Solymosi15c} again to get many $(7,4)$-configurations using only these vertices.

\begin{figure}[ht!]
	\centering
	\includegraphics[width=120mm]{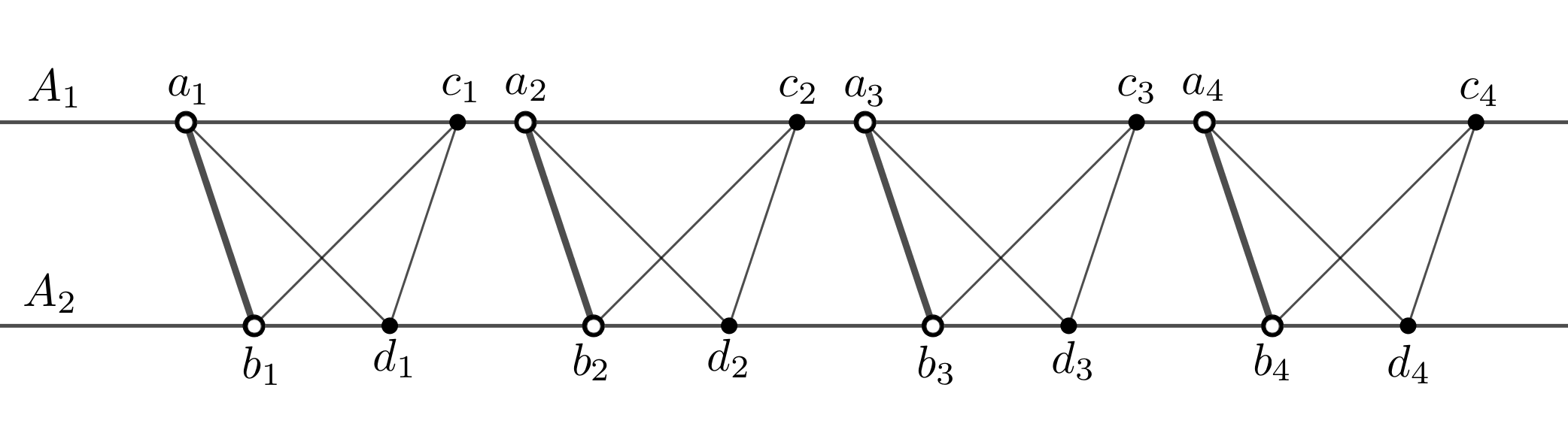}
	\caption{The thicker edges are in $\cE$ and the vertices represented by empty circles are in $\cV$.}
	\label{fig:E,V}
\end{figure}

We note that the size of the set
\[
\cE = \cE(S_0',\vec y^{(1)}) := \{ (a,b) \in  A_1 \times A_2 : (a,b,c,d) \in \tilde{Q}_{S_0'}(\vec{y}^{(1)})\} \subset E(G_{S_0'})
\]
is $\tilde q_{S_0'}(\vec y^{(1)})$. Denote by $\cV(S_0',\vec y^{(1)}) \subset V(G_{S_0'})$ the set of vertices from $\cE(S_0',\vec y^{(1)})$, i.e.
\[
\cV = \cV(S_0',\vec y^{(1)}) := \{a \in A_1 : (a,b) \in \cE(S_0',\vec y^{(1)}) \} \cup \{b \in A_2 : (a,b) \in \cE(S_0',\vec y^{(1)}) \}.
\]
See Figure \ref{fig:E,V}.

\begin{lemma}
	\label{lemma:epsregular}
	If $\eps = \min (\frac {\lfloor \mu_1(c/2) \rfloor} 6, \frac c {24})$, then there exist indices $1 \leq i,j \leq k$ such that $|V_i \cap \cV| \geq \eps |V_i|$, $|V_{k+j} \cap \cV| \geq \eps |V_{k+j}|$, and there is an edge in $\cE$ between $V_i$ and $V_{k+j}$.
\end{lemma}

\begin{proof}
	\label{lemma:eps}
	For each $1 \leq i \leq 2k$, we put 
	\[
	\cE_i = \begin{cases}
	\emptyset, &\mbox{if $|V_i \cap \cV| \geq \eps |V_i|$} \\
	\{(a,b) \in \cE : a \in V_i \cap \cV\}, &\mbox{if $|V_i \cap \cV| < \eps |V_i|$ and $1 \leq i \leq k$} \\
	\{(a,b) \in \cE : b \in V_i \cap \cV\}, &\mbox{if $|V_i \cap \cV| < \eps |V_i|$ and $k+1 \leq i \leq 2k$}
	\end{cases} \subset \cE.
	\]
	
	It suffices to show that $\cE \backslash (\cup_{i=1}^{2k} \cE_i)$ is non-empty. Recall from condition \emph{(iii)} that $|V_i| \leq n/k$. Since $|\cE_i| < \eps |V_i| \leq \eps n / k$ and $\eps \leq \lfloor \mu_1(c/2) \rfloor/6$, we have
	\[
	\left| \cE \bigg\backslash \bigcup_{i=1}^{2k} \cE_i \right| \geq |\cE| - \sum_{i=1}^{2k} |\cE_i| > \tilde q_{S_0'}(\vec y^{(1)}) - 2\eps n > \dfrac{\lfloor \mu_1(c/2) n \rfloor}{3} - \dfrac{\lfloor \mu_1(c/2) \rfloor n}{3} = 0.
	\]
\end{proof}

By condition \emph{(vi)} and \emph{(i)}, the number of edges between $V_i \cap \cV$ and $V_{k+j} \cap \cV$ in $G_{S_0'}$ is at least
\[
\eps |V_i \cap \cV| |V_{k+j} \cap \cV| \geq \eps^3 |V_i||V_{k+j}| \geq \dfrac{\eps^3(1-\eps)^2n^2}{k^2} \geq \dfrac{4\eps^3(1-\eps)^2}{K^2} n^2.
\]

Remove from these edges the set of edges $(\alpha,\beta)$ so that $\alpha\beta \in \{y_1^{(1)},y_2^{(1)},y_3^{(1)}\}$, where $\vec y^{(1)} = (y_1^{(1)}, y_2^{(1)}, y_3^{(1)})$. The number of edges removed is at most $3|V_i| \leq 3n/k \leq 6 \eps n$, since $A$ is a group. Therefore, the number of remaining edges is at least 
\[
\dfrac{4\eps^3(1-\eps)^2}{K^2} n^2 - 6 \eps n \geq \dfrac{\eps^3(1-\eps)^2}{K^2} n^2 =: c_1n^2,
\] 
where $c_1 > 0$ depends only on $c$, assuming $n \geq 2K^2 / (\eps(1-\eps))^2$.

The subgraph of $G_{S_0'}$ containing these edges can be written as $G_{S_1}$, for some $S_1 \subset S_0'$ with $|S_1| \geq c_1n^2$. Apply \cref{thm:Solymosi15c} with $S = S_1$, we get $\lfloor \mu_1(c_1) n^3 \rfloor$ $S_1$-good quadruples in $A^4$.

Let $(a,b,c,d) \in Q_{S_1}(\vec y^{(2)})$ be one of these $S_1$-good quadruples. Since the graph $G_{S_1}$ does not contain an edge $(\alpha,\beta)$ with $\alpha\beta \in \{y_1^{(1)},y_2^{(1)},y_3^{(1)}\}$, the triples $\vec y^{(1)}$ and $\vec y^{(2)}$ are element-disjoint. Moreover, since $a,c \in V_i \cap \cV \subset \cV \cap A_1$, there are $b_1,b_3 \in A_2$ such that $(a,b_1),(c,b_3) \in \cE$, and so we have $(a,b_1,c_1,d_1),(c,b_3,c_3,d_3) \in \tilde{Q}_{S_0'}(\vec y^{(1)})$. Similarly, since $b,d \in V_{k+j} \cap \cV$, we have $(a_2,b,c_2,d_2),(a_4,d,c_4,d_4) \in \tilde{Q}_{S_0'}(\vec y^{(1)})$. The claim we made in the beginning of this section is now established.

Recall that we imposed 2 assumptions on $n$, namely $n \geq n_0(\eps)$ and $n \geq 2K^2 / (\eps(1-\eps))^2$. Hence, we put
\[
\mu_2(c) = \min \left( \mu_1(c_1), \left(\dfrac{1}{n_0(\eps)}\right)^3, \left( \dfrac{\eps^2(1-\eps)^2}{2K^2} \right)^3 \right).
\]

\section{Proof of \cref{thm:mu} when $t \geq 3$}
To prove \cref{thm:mu} for larger $t$, we repeat what we did in the last section $t-2$ times. We highlight the key points when $t = 3$, assuming $n$ is large enough.

After we get $S_1$, instead of applying \cref{thm:Solymosi15c} immediately, we first apply the regularity lemma to $G_{S_1}$ with $\eps = \min (\frac {\lfloor \mu_1(c_1/2) \rfloor} 6, \frac {c_1} {24})$ to get $S_1' \subset S_1$ with $|S_1'| \geq \frac{c_1}{2}n^2$. By \cref{lemma:largestqvalue}, we get a triple $\vec y^{(2)} \in A^3$ such that 
\[
\tilde{q}_{S_1'}(\vec{y}^{(2)}) > \dfrac{\lfloor \mu_1(c_1/2) n \rfloor}{3}.
\]
Define the sets $\cE(S_1',\vec y^{(2)})$ and $\cV(S_1',\vec y^{(2)})$ accordingly. In the same spirit, \cref{lemma:epsregular} allows us to find a subset $S_2 \subset S_1'$ with $|S_2| \geq c_2n^2$, where $c_2 > 0$ depends on $c$. Note that the graph $G_{S_2}$ does not contain an edge $(\alpha,\beta)$ with $\alpha \beta \in \{y_1^{(2)},y_2^{(2)},y_3^{(2)}\}$. Lastly, we apply \cref{thm:Solymosi15c} and get $\lfloor \mu_1(c_2) n^3 \rfloor$ $S_2$-good quadruples in $A^4$.

In general, we get subsets $S_{t-1} \subset S_{t-2}' \subset S_{t-2} \subset S_{t-3}' \subset \cdots \subset S_0' \subset S_0$, with $|S_i'| \geq \frac{c_i}{2} n^2$ and $|S_i| \geq c_i n^2$, where $c_i > 0$ depends on $c$.  

It remains to show that each $S_{t-1}$-good quadruple $(a,b,c,d) \in A^4$ gives us a set of $\nu$ elements, with $2^{t+1} \leq \nu \leq 4^t + 3t$, which spans at least
\[
\dfrac{4(\nu-3t)}{3} \left(1 - \dfrac{1}{4^t} \right)
\]
triples of $S_0$.

For $1 \leq i \leq t$, let $\cL_i$ be the set of good quadruples in the $i$-th \emph{layer}. More precisely, we define
\[
\cL_1 := \{(a,b,c,d)\} \subset Q_{S_{t-1}}(\vec y^{(t)}),
\]
and if we have $\cL_i = \{(a_j,b_j,c_j,d_j) : 1 \leq j \leq |\cL_i| \} \subset Q_{S_{t-i}'}(\vec y^{(t-i+1)})$, one can define $\cL_{i+1} \subset Q_{S_{t-i-1}'}(\vec y^{(t-i)})$ as follows. Since $a_j,c_j \in \cV(S_{t-i-1}',\vec y^{(t-i)}) \cap A_1$, we have $(a_j,\beta_{j,1},\gamma_{j,1},\delta_{j,1}), (c_j,\beta_{j,3},\gamma_{j,3},\delta_{j,3}) \in Q_{S_{t-i-1}'}(\vec y^{(t-i)})$. Similarly, since $b_j,d_j \in \cV(S_{t-i-1}',\vec y^{(t-i)}) \cap A_2$, we have $(\alpha_{j,2},b_j,\gamma_{j,2},\delta_{j,2}), (\alpha_{j,4},d_j,\gamma_{j,4},\delta_{j,4}) \in Q_{S_{t-i-1}'}(\vec y^{(t-i)})$. Put
\[
\begin{split}
\cL_{i+1} = \{(a_j,\beta_{j,1},&\gamma_{j,1},\delta_{j,1}), (c_j,\beta_{j,3},\gamma_{j,3},\delta_{j,3}),(\alpha_{j,2},b_j,\gamma_{j,2},\delta_{j,2}), \\
&(\alpha_{j,4},d_j,\gamma_{j,4},\delta_{j,4}): 1 \leq j \leq |\cL_i|\} \subset Q_{S_{t-i-1}'}(\vec y^{(t-i)})
\end{split}
\]

 For $2 \leq i \leq t$, the set $\cL_i \subset \tilde Q_{S_{t-i}'}(\vec y^{(t-i+1)})$ consists of $|\cL_i| \leq 4|\cL_{i-1}|$ $S_{t-i}'$-good quadruples. Note also that the set of elements of $A$ appear in $\cL_t$ contains all the elements appear in other layers.

\begin{lemma}
	\label{lemma:L_i}
	For $2 \leq i \leq t$, the number of distinct elements of $A$ appear in $\cL_i$ is at least $2^{i+1}$ and at most $4^i$.
\end{lemma}

\begin{proof}
	Since we have
	\[
	|\cL_i| \leq 4|\cL_{i-1}| \leq \cdots \leq 4^{i-1}|\cL_1| = 4^{i-1},
	\]
	and since each quadruple contains at most 4 elements of $A$, the upper bound follows.
	
	Since the quadruples in $\cL_i$ are pairwise disjoint, the elements
	\[
	\{a : (a,b,c,d) \in \cL_i \} \cup \{c : (a,b,c,d) \in \cL_i \}
	\]
	are all distinct. Hence, the number of distinct elements of $A$ appear in $\cL_i$ is at least $2|\cL_i|$. Recall that we showed in \cref{section:t=2} that $|\cL_2| = 4$. To prove the lower bound, it suffices to show that $|\cL_{i+1}| \geq 2|\cL_i|$, for all $i \geq 2$.
	
	To this end, we recall from above that the set
	\[
	\{(a_j,\beta_{j,1},\gamma_{j,1},\delta_{j,1}) : 1 \leq j \leq |\cL_i|\} \cup \{(c_j,\beta_{j,3},\gamma_{j,3},\delta_{j,3}): 1 \leq j \leq |\cL_i|\} 
	\]
	is a subset of $\cL_{i+1}$. As we noted that the elements $\{a_j,c_j : 1 \leq j \leq |\cL_i|\}$ are all distinct, no two quadruples above can be the same. This proves the lower bound.
\end{proof}

Now, we consider the elements of $A$ appear in $\cL_t$ and all the triples $\vec y^{(j)}$, where $1 \leq j \leq t$. The total number $\nu$ of elements satisfies $2^{t+1} \leq \nu \leq 4|\cL_t| + 3t \leq 4^t + 3t$. The number of triples of $S_0$ spanned by these $\nu$ elements is at least
\[
\begin{split}
4(|\cL_1| + \cdots + |\cL_{t-1}| + |\cL_t|)
& \geq 4|\cL_t| \left(\dfrac{1}{4^{t-1}} + \cdots + \dfrac14 + 1\right) \\
& = \dfrac{4|\cL_t|}{3} \dfrac{4^t-1}{4^{t-1}} \\
& \geq \dfrac{4(\nu - 3t)}{3} \left( 1-\dfrac{1}{4^t} \right),
\end{split}
\]
as desired.

\section{Acknowledgements}
Both authors were supported in part by NSERC. The work of the first author was supported by the European Research Council
(ERC) under the European Union's Horizon 2020 research and innovation programme (grant agreement No. 741420, 617747, 648017).
The first author was also supported by OTKA K 119528.

\end{document}